\documentclass{amsart}%
\usepackage{amsmath}
\usepackage{amssymb}
\usepackage{amsfonts}
\usepackage{graphicx}%
\providecommand{\U}[1]{\protect\rule{.1in}{.1in}}
\newtheorem{theorem}{Theorem}[section]
\theoremstyle{plain}

\newtheorem{corollary}{Corollary}[section]

\newtheorem{example}{Example}[section]

\newtheorem{lemma}{Lemma}[section]

\numberwithin{equation}{section}
\begin{document}
\title[Stern polynomials and algebraic independence]{Stern polynomials and algebraic independence}
\author{Daniel Duverney}
\address{110, rue du chevalier fran\c{c}ais, 59000 Lille, France}
\email{daniel.duverney@orange.fr}
\author{Iekata Shiokawa}
\address{13-43, Fujizuka-cho, Hodogaya-ku, Yokohama 240-0031, Japan}
\email{shiokawa@beige.ocn.ne.jp}
\date{March 23, 2026}
\maketitle

\begin{abstract}
Let $t\geq2$ and $k\geq1$ be integers. Let $H_{k}(z)$ with $\left\vert
z\right\vert <1$ be the limit of a certain subsequence of the Stern
polynomials introduced by Dilcher and Eriksen. We use Mahler's method to prove
the algebraic independence of the values at nonzero algebraic points of the
functions $H_{k}(z)$ and $H_{k}(z^{t^{k}})$.

\end{abstract}

\section{Introduction}

Throughout this paper, let $t\geq2$ and $k\geq1$ be any fixed integers.
Stern's diatomic sequence $a(n),$ defined by $a(0)=0,$ $a(1)=1,$ and for
$n\geq1$%
\[
a(2n)=a(n),\quad a(2n+1)=a(n+1)+a(n),
\]
has been studied by many authors (cf. Sequence OEIS A002487 in \cite{O}). In
2007, Klavzar, Milutinovic and Petr introduced Stern polynomials $a(n;z)$
defined by $a(0;z)=0,$ $a(1;z)=1$ and for $n\geq1$%
\[
a(2n;z)=za(n;z),\quad a(2n+1;z)=a(n+1;z)+a(n;z).
\]
In 2018, Dilcher and Eriksen \cite{Dilcher} extended them to (Type 1) Stern
polynomials $a_{t}(n;z)$ defined by $a_{t}(0;z)=0,$ $a_{t}(1;z)=1$ and for
$n\geq1$%
\begin{align}
a_{t}(2n;z)  &  =za_{t}(n;z^{t}),\quad\label{Stern}\\
a_{t}(2n+1;z)  &  =a_{t}(n+1;z^{t})+a_{t}(n;z^{t}) \label{SternB}%
\end{align}
whose coefficients are only $0$ or $1$ (they also defined in \cite[Sec.5]%
{Dilcher} another 'Type 2' Stern polynomials, which however will not be
discussed in this paper). They introduced the sequence of integers%
\[
\alpha_{n}=\alpha_{n}(k)=\frac{2^{kn}-\left(  -1\right)  ^{n}}{2^{k}+1}%
\quad\left(  n\geq0\right)
\]
and proved \cite[Prop.6.1]{Dilcher} that the subsequence $\left(  a_{t}%
(\alpha_{n};z)\right)  _{n\geq0}$ satisfies the three-term recurrence relation%
\begin{equation}
a_{t}(\alpha_{n+1};z)=a_{t}(2^{k}-1;z)a_{t}(\alpha_{n};z^{t^{k}})+a_{t}%
(2^{k};z^{t^{k}})a_{t}(\alpha_{n-1};z^{t^{2k}}), \label{Rec}%
\end{equation}
which generates \cite[Prop.6.2]{Dilcher} the finite continued fraction%
\begin{align}
&  \frac{a_{t}(\alpha_{n+1};z)}{a_{t}(\alpha_{n};z^{t^{k}})}\left.  =\right.
a_{t}(2^{k}-1;z)+\frac{a_{t}(2^{k};z^{t^{k}})}{a_{t}(2^{k}-1;z^{t^{k}}%
)}\label{CF1}\\
&  \qquad\qquad\qquad\qquad\qquad\qquad%
\genfrac{}{}{0pt}{}{{}}{+}%
\frac{a_{t}(2^{k};z^{t^{2k}})}{a_{t}(2^{k}-1;z^{t^{2k}})}%
\genfrac{}{}{0pt}{}{{}}{+\cdots+}%
\frac{a_{t}(2^{k};z^{t^{\left(  n-1\right)  k}})}{a_{t}(2^{k}-1;z^{t^{\left(
n-1\right)  k}})}.\nonumber
\end{align}
The polynomials $a_{t}(\alpha_{n+1};z)$ and $a_{t}(\alpha_{n};z)$ agree to
ever higher powers of $z,$ which proves \cite[Prop.6.4]{Dilcher} the existence
of the limit%
\[
H_{k}(z)=\lim_{n\rightarrow\infty}a_{t}(\alpha_{n};z),
\]
which defines a power series converging in $\left\vert z\right\vert <1$ and
having only $0$ or $1$ as coefficients. Letting $n\rightarrow\infty$ in
$\left(  \ref{Rec}\right)  $ and $\left(  \ref{CF1}\right)  ,$ we find that
$H_{k}(z)$ satisfies the functional equation%
\begin{equation}
H_{k}(z)=a_{t}(2^{k}-1;z)H_{k}(z^{t^{k}})+a_{t}(2^{k};z^{t^{k}})H_{k}%
(z^{t^{2k}}). \label{Rec2}%
\end{equation}
This generates \cite[Prop.6.5]{Dilcher} the infinite continued fraction%
\begin{align}
&  \frac{H_{k}(z)}{H_{k}(z^{t^{k}})}\left.  =\right.  a_{t}(2^{k}%
-1;z)+\frac{a_{t}(2^{k};z^{t^{k}})}{a_{t}(2^{k}-1;z^{t^{k}})}\label{CF2}\\
&  \qquad\qquad\qquad\qquad\qquad%
\genfrac{}{}{0pt}{}{{}}{+}%
\frac{a_{t}(2^{k};z^{t^{2k}})}{a_{t}(2^{k}-1;z^{t^{2k}})}%
\genfrac{}{}{0pt}{}{{}}{+\cdots+}%
\frac{a_{t}(2^{k};z^{t^{nk}})}{a_{t}(2^{k}-1;z^{t^{nk}})}%
\genfrac{}{}{0pt}{}{{}}{+\cdots}%
,\nonumber
\end{align}
which converges in $\left\vert z\right\vert <1,$ where by \cite[Lem.6.1]%
{Dilcher}%
\begin{equation}
a_{t}(2^{k};z)=z^{\frac{t^{k}-1}{t-1}},\quad a_{t}(2^{k}-1;z)=\sum_{i=1}%
^{k}z^{\frac{t^{k}-t^{i}}{t-1}}. \label{at1}%
\end{equation}
By $\left(  \ref{Rec2}\right)  ,$ $f_{1}(z)=H_{k}(z)$ and $f_{2}%
(z)=H_{k}(z^{t^{k}})$ satisfy the system of functional equations%
\begin{equation}
\left(
\begin{array}
[c]{c}%
f_{1}(z^{t^{k}})\\
f_{2}(z^{t^{k}})
\end{array}
\right)  =A(z)\left(
\begin{array}
[c]{c}%
f_{1}(z)\\
f_{2}(z)
\end{array}
\right)  , \label{Mat}%
\end{equation}
where%
\begin{align}
A(z)  &  =\frac{1}{a_{t}(2^{k};z^{t^{k}})}B(z),\quad\label{Mat1}\\
B(z)  &  =\left(
\begin{array}
[c]{cc}%
0 & a_{t}(2^{k};z^{t^{k}})\\
1 & -a_{t}(2^{k}-1;z)
\end{array}
\right)  . \label{Mat2}%
\end{align}

Our main result is

\begin{theorem}
\label{ThStern}For any algebraic number $\alpha$ with $0<\left\vert
\alpha\right\vert <1,$ the numbers $H_{k}(\alpha)$ and $H_{k}(\alpha^{t^{k}})$
are algebraically independent.
\end{theorem}

As an application, we get

\begin{corollary}
\label{CorStern}The continued fraction%
\[
\frac{H_{k}(\alpha)}{H_{k}(\alpha^{t^{k}})}\left.  =\right.  a_{t}%
(2^{k}-1;\alpha)+\frac{a_{t}(2^{k};\alpha^{t^{k}})}{a_{t}(2^{k}-1;\alpha
^{t^{k}})}%
\genfrac{}{}{0pt}{}{{}}{+\cdots+}%
\frac{a_{t}(2^{k};\alpha^{t^{nk}})}{a_{t}(2^{k}-1;\alpha^{t^{nk}})}%
\genfrac{}{}{0pt}{}{{}}{+\cdots}%
\]
is transcendental for algebraic $\alpha$ with $0<\left\vert \alpha\right\vert
<1,$
\end{corollary}

\begin{example}
The following continued fractions are transcendental.%
\begin{align*}
\frac{H_{1}(\alpha)}{H_{1}(\alpha^{t})}  &  =1+\frac{\alpha^{t}}{1}%
\genfrac{}{}{0pt}{}{{}}{+}%
\frac{\alpha^{t^{2}}}{1}%
\genfrac{}{}{0pt}{}{{}}{+}%
\frac{\alpha^{t^{3}}}{1}%
\genfrac{}{}{0pt}{}{{}}{+\cdots}%
,\\
\frac{H_{2}(\alpha)}{H_{2}(\alpha^{t^{2}})}  &  =1+\alpha^{t}+\frac
{\alpha^{t^{2}+t^{3}}}{1+\alpha^{t^{3}}}%
\genfrac{}{}{0pt}{}{{}}{+}%
\frac{\alpha^{t^{4}+t^{5}}}{1+\alpha^{t^{5}}}%
\genfrac{}{}{0pt}{}{{}}{+}%
\frac{\alpha^{t^{6}+t^{7}}}{1+\alpha^{t^{7}}}%
\genfrac{}{}{0pt}{}{{}}{+\cdots}%
,\\
\frac{H_{3}(\alpha)}{H_{3}(\alpha^{t^{3}})}  &  =1+\alpha^{t^{2}}%
+\alpha^{t+t^{2}}+\frac{\alpha^{t^{3}+t^{4}+t^{5}}}{1+\alpha^{t^{5}}%
+\alpha^{t^{4}+t^{5}}}%
\genfrac{}{}{0pt}{}{{}}{+}%
\frac{\alpha^{t^{6}+t^{7}+t^{8}}}{1+\alpha^{t^{8}}+\alpha^{t^{7}+t^{8}}}%
\genfrac{}{}{0pt}{}{{}}{+\cdots}%
.
\end{align*}

\end{example}

Our proof of Theorem \ref{ThStern} in Section \ref{Sec2} is an application of
Mahler's transcendence method \cite{Nishi} to continued fractions, as
previously developed by Ku. Nishioka \cite[Section 5.2]{Nishi} and Adamczewski
\cite{Adam}.

Another application of Mahler's method to the transcendence of continued
fractions has been proposed by Bundschuh and V\"{a}\"{a}n\"{a}nen, and
Corollary \ref{CorStern} above could be also deduced from \cite[Th.1.1]{Bund}.

\section{Proof of Theorem \ref{ThStern}}

\label{Sec2}For the proof we prepare four lemmas.

\begin{lemma}
\label{LemStern}Let $A(z)$ be defined in $\left(  \ref{Mat1}\right)  .$ Then
$0 $ is the only pole of $A(z).$
\end{lemma}

\begin{proof}
Immediate consequence of $\left(  \ref{at1}\right)  .$
\end{proof}

\begin{lemma}
\label{LemStern1}$H_{k}(1/2)/H_{k}(1/2^{t^{k}})$ is irrational.
\end{lemma}

\begin{proof}
By $\left(  \ref{at1}\right)  ,$ we have%
\begin{equation}
d_{n}:=\deg a_{t}(2^{k};z^{t^{nk}})=\deg a_{t}(2^{k}-1;z^{t^{nk}}%
)+t^{nk}=t^{nk}\frac{t^{k}-1}{t-1}. \label{Deg}%
\end{equation}
Replacing $z$ by $1/2$ in $\left(  \ref{CF2}\right)  $ yields
\begin{align*}
\frac{H_{k}(1/2)}{H_{k}(1/2^{t^{k}})}  &  =a_{t}(2^{k}-1;1/2)+\frac
{a_{t}(2^{k};1/2^{t^{k}})}{a_{t}(2^{k}-1;1/2^{t^{k}})}%
\genfrac{}{}{0pt}{}{{}}{+}%
\frac{a_{t}(2^{k};1/2^{t^{2k}})}{a_{t}(2^{k}-1;1/2^{t^{2k}})}\\
&  \qquad%
\genfrac{}{}{0pt}{}{{}}{+}%
\frac{a_{t}(2^{k};1/2^{t^{3k}})}{a_{t}(2^{k}-1;1/2^{t^{3k}})}%
\genfrac{}{}{0pt}{}{{}}{+\cdots+}%
\frac{a_{t}(2^{k};1/2^{t^{nk}})}{a_{t}(2^{k}-1;1/2^{t^{nk}})}%
\genfrac{}{}{0pt}{}{{}}{+\cdots}%
\\
&  =a_{t}(2^{k}-1;1/2)+\frac{2^{d_{1}}a_{t}(2^{k};1/2^{t^{k}})}{2^{d_{1}}%
a_{t}(2^{k}-1;1/2^{t^{k}})}\\
&  \qquad%
\genfrac{}{}{0pt}{}{{}}{+}%
\frac{2^{d_{2}}a_{t}(2^{k};1/2^{t^{2k}})}{2^{d_{2}-d_{1}}a_{t}(2^{k}%
-1;1/2^{t^{2k}})}%
\genfrac{}{}{0pt}{}{{}}{+\cdots+}%
\frac{2^{d_{n}}a_{t}(2^{k};1/2^{t^{nk}})}{2^{D_{n}}a_{t}(2^{k}-1;1/2^{t^{nk}%
})}%
\genfrac{}{}{0pt}{}{{}}{+\cdots}%
,\qquad
\end{align*}
where $D_{n}=d_{n}-d_{n-1}+d_{n-2}-\cdots+\left(  -1\right)  ^{n-1}d_{1}.$

We note that $2^{d_{n}}a_{t}(2^{k};1/2^{t^{nk}})=1$ and%
\[
2^{D_{n}}a_{t}(2^{k}-1;1/2^{t^{nk}})\in\mathbb{Z}_{>0},
\]
since $D_{n}-\deg a_{t}(2^{k}-1;z^{t^{nk}})>0$ for all $n\geq1.$ Indeed, this
is true for $n=1$ by $\left(  \ref{Deg}\right)  $ since $D_{1}=d_{1},$ and
true also by $\left(  \ref{Deg}\right)  $ for $n\geq2$ since%
\begin{gather*}
D_{n}-\deg a_{t}(2^{k}-1;z^{t^{nk}})\geq d_{n}-d_{n-1}-\left(  d_{n}%
-t^{nk}\right) \\
=t^{nk}-t^{(n-1)k}\frac{t^{k}-1}{t-1}>0.
\end{gather*}
Hence $H_{k}(1/2)/H_{k}(1/2^{t^{k}})$ is an infinite regular continued
fraction, which is known to be irrational.
\end{proof}

\begin{lemma}
\label{LemStern2}$H_{k}(z)$ is transcendental over $\mathbb{C}(z).$
\end{lemma}

\begin{proof}
We appeal to Fatou's theorem \cite{Fatou}: \textit{A power series whose
coefficients take only finitely many values is rational or transcendental.} As
the coefficients of $H_{k}(z)$ are only $0$ or $1,$ $H_{k}(z)$ is rational or
transcendental. If $H_{k}(z)\in\mathbb{C}\left(  z\right)  ,$ then
$H_{k}(\alpha)\in\mathbb{Q}$ for each $\alpha\in\mathbb{Q}$ since $H_{k}%
(z)\in\mathbb{Q}\left[  \left[  z\right]  \right]  ,$ and so $H_{k}%
(1/2)/H_{k}(1/2^{t^{k}})\in\mathbb{Q}$, which contradicts Lemma
\ref{LemStern1}. Therefore $H_{k}(z)\notin\mathbb{C}\left(  z\right)  $ and is transcendental.
\end{proof}

\begin{lemma}
\label{LemStern3}The power series $f_{1}(z)=H_{k}(z)$ and $f_{2}%
(z)=H_{k}(z^{t^{k}})$ are algebraically independent over $\mathbb{C}(z). $
\end{lemma}

\begin{proof}
We apply \cite[Th. 5.2]{Nishi}. By $\left(  \ref{Mat}\right)  $ and Lemma
\ref{LemStern2}, the assumptions of the theorem are satisfied. Assume that
$f_{1}(z)$ and $f_{2}(z)$ are algebraically dependent over $\mathbb{C}\left(
z\right)  .$ Then, setting%
\begin{equation}
G^{(n)}(z)=\left(
\begin{array}
[c]{cc}%
G_{11}^{(n)}(z) & G_{12}^{(n)}(z)\\
G_{21}^{(n)}(z) & G_{22}^{(n)}(z)
\end{array}
\right)  =B(z^{t^{(n-1)k}})\cdots B(z^{t^{k}})B(z), \label{MatG}%
\end{equation}
there exists $n_{0}\in\mathbb{Z}_{>0}$ such that at least one of the following
three conditions holds.

(i) $G_{12}^{(n)}(z)=0$ for $n=n_{0},2n_{0},3n_{0},$\ldots

(ii) $G_{21}^{(n)}(z)=0$ for $n=n_{0},2n_{0},3n_{0},$ \ldots

(iii) There exist $r\in\mathbb{Z}_{>0}$ and coprime $a(z),b(z)\in$
$\mathbb{C}\left[  z\right]  $ satisfying%
\begin{equation}
\frac{b(z)}{a(z)}=\frac{b(z^{t^{nk}})G_{11}^{(n)}(z^{r})+a(z^{t^{nk}}%
)G_{21}^{(n)}(z^{r})}{b(z^{t^{nk}})G_{12}^{(n)}(z^{r})+a(z^{t^{nk}}%
)G_{22}^{(n)}(z^{r})} \label{Quot}%
\end{equation}
for $n=n_{0},2n_{0},3n_{0},\ldots.$

We will prove that none of the properties (i)-(iii) is fulfilled, and Lemma
\ref{LemStern3} will follow by contradiction.\smallskip

Let us prove first that neither the property (i) nor (ii) is fulfilled. By
$\left(  \ref{MatG}\right)  $ and $\left(  \ref{Mat2}\right)  $ we have%
\begin{align}
G_{11}^{(1)}(z)  &  =0,\text{ }G_{12}^{(1)}(z)=a_{t}(2^{k};z^{t^{k}}),\text{
}\label{G0}\\
G_{21}^{(1)}(z)  &  =1,\text{ }G_{22}^{(1)}(z)=-a_{t}(2^{k}-1;z), \label{G10}%
\end{align}
and for $n\geq1$%
\begin{align}
&  G_{11}^{(n+1)}(z)\left.  =\right.  a_{t}(2^{k};z^{t^{(n+1)k}})G_{21}%
^{(n)}(z),\quad\label{G1}\\
&  G_{12}^{(n+1)}(z)\left.  =\right.  a_{t}(2^{k};z^{t^{(n+1)k}})G_{22}%
^{(n)}(z),\label{G2}\\
&  G_{21}^{(n+1)}(z)\left.  =\right.  G_{11}^{(n)}(z)-a_{t}(2^{k}-1;z^{t^{nk}%
})G_{21}^{(n)}(z),\label{G3}\\
&  G_{22}^{(n+1)}(z)\left.  =\right.  G_{12}^{(n)}(z)-a_{t}(2^{k}-1;z^{t^{nk}%
})G_{22}^{(n)}(z). \label{G4}%
\end{align}

Since $a_{t}(2^{k};0)=0$ and $a_{t}(2^{k}-1;0)=1$ by $\left(  \ref{at1}%
\right)  ,$ we get from (\ref{G0})-(\ref{G4})%
\begin{align}
G_{11}^{(n)}(0)  &  =G_{12}^{(n)}(0)=0\quad\quad\left(  n\geq1\right)
,\label{G5}\\
G_{21}^{(n)}(0)  &  =-G_{22}^{(n)}(0)=\left(  -1\right)  ^{n-1}\quad
\quad\left(  n\geq1\right)  . \label{G6}%
\end{align}

By (\ref{G6}), $G_{21}^{(n)}(z)\neq0$ and $G_{22}^{(n)}(z)\neq0$ for $n\geq1.
$ Hence this yields $G_{12}^{(n)}(z)\neq0$ for $n\geq2$ by (\ref{G2}), and so
(i) and (ii) cannot hold.\smallskip

Now assume that (iii) is fulfilled. Since $a(z)$ and $b(z)$ are coprime, we
see by $\left(  \ref{Quot}\right)  $ that there exists $\kappa^{(n)}%
(z)\in\mathbb{C}\left[  z\right]  $ such that%
\begin{align}
b(z^{t^{nk}})G_{11}^{(n)}(z^{r})+a(z^{t^{nk}})G_{21}^{(n)}(z^{r})  &
=\kappa^{(n)}(z)b(z),\label{Eq1}\\
b(z^{t^{nk}})G_{12}^{(n)}(z^{r})+a(z^{t^{nk}})G_{22}^{(n)}(z^{r})  &
=\kappa^{(n)}(z)a(z). \label{Eq2}%
\end{align}
Thus by Cramer's formulas%
\begin{align*}
b(z^{t^{nk}})\Delta_{n}\left(  z\right)   &  =\kappa^{(n)}(z)\left\vert
\begin{array}
[c]{cc}%
b(z) & G_{21}^{(n)}(z^{r})\\
a(z) & G_{22}^{(n)}(z^{r})
\end{array}
\right\vert ,\\
a(z^{t^{nk}})\Delta_{n}\left(  z\right)   &  =\kappa^{(n)}(z)\left\vert
\begin{array}
[c]{cc}%
G_{11}^{(n)}(z^{r}) & b(z)\\
G_{12}^{(n)}(z^{r}) & a(z)
\end{array}
\right\vert ,
\end{align*}
where%
\[
\Delta_{n}\left(  z\right)  :=\left\vert
\begin{array}
[c]{cc}%
G_{11}^{(n)}(z^{r}) & G_{21}^{(n)}(z^{r})\\
G_{12}^{(n)}(z^{r}) & G_{22}^{(n)}(z^{r})
\end{array}
\right\vert .
\]
By $\left(  \ref{MatG}\right)  $ and $\left(  \ref{Mat2}\right)  ,$ we have%
\begin{align*}
\Delta_{n}\left(  z\right)   &  =\det\left(  B(z^{rt^{(n-1)k}})\cdots
B(z^{rt^{k}})B(z^{r})\right) \\
&  =\left(  -1\right)  ^{n}a_{t}(2^{k};z^{rt^{kn}})\cdots a_{t}(2^{k}%
;z^{rt^{k}}).
\end{align*}
Hence by $\left(  \ref{at1}\right)  $
\[
\Delta_{n}\left(  z\right)  =\left(  -1\right)  ^{n}z^{\delta_{n}},
\]
where $\delta_{n}\in\mathbb{Z}_{>0}$. So if $\kappa^{(n)}(\alpha)=0$ for some
$\alpha\neq0,$ then $a(\alpha^{t^{nk}})=b(\alpha^{t^{nk}})=0.$ This is
impossible since $a(z)$ and $b(z)$ are coprime. Consequently
\begin{equation}
\kappa^{(n)}(z)=c_{n}z^{u_{n}},\quad c_{n}\in\mathbb{C}^{\ast},\quad0\leq
u_{n}\leq\delta_{n}. \label{Gam}%
\end{equation}
We distinguish two cases.

\textbf{First case.\quad}If $a(0)\neq0,$ then putting $z=0$ in (\ref{Eq2}), we
get $\kappa^{(n)}(0)=\left(  -1\right)  ^{n}$ by (\ref{G5}) and (\ref{G6}). By
(\ref{Gam}), this yields
\begin{equation}
\kappa^{(n)}(z)=\left(  -1\right)  ^{n}\qquad\left(  a(0)\neq0\right)  .
\label{Case1}%
\end{equation}

\textbf{Second case.\quad}Let $a(0)=0.$ Then%
\[
\left(
\begin{array}
[c]{cc}%
G_{11}^{(n)}(z^{r}) & G_{21}^{(n)}(z^{r})\\
G_{12}^{(n)}(z^{r}) & G_{22}^{(n)}(z^{r})
\end{array}
\right)  \left(
\begin{array}
[c]{c}%
b(z^{t^{nk}})\\
a(z^{t^{nk}})
\end{array}
\right)  =\kappa^{(n)}(z)\left(
\begin{array}
[c]{c}%
b(z)\\
a(z)
\end{array}
\right)
\]
by (\ref{Eq1}) and (\ref{Eq2}), whence%
\begin{equation}
\frac{\Delta_{n}(z)}{\kappa^{(n)}(z)}\left(
\begin{array}
[c]{c}%
b(z^{t^{nk}})\\
a(z^{t^{nk}})
\end{array}
\right)  =\left(
\begin{array}
[c]{cc}%
G_{22}^{(n)}(z^{r}) & -G_{21}^{(n)}(z^{r})\\
-G_{12}^{(n)}(z^{r}) & G_{11}^{(n)}(z^{r})
\end{array}
\right)  \left(
\begin{array}
[c]{c}%
b(z)\\
a(z)
\end{array}
\right)  . \label{Delt}%
\end{equation}
Taking (\ref{Gam}) into account, let us define%
\begin{equation}
C^{(n)}(z):=\frac{\Delta_{n}(z)}{\kappa^{(n)}(z)}=\frac{\left(  -1\right)
^{n}}{c_{n}}z^{\delta_{n}-u_{n}}=\frac{\left(  -1\right)  ^{n}}{c_{n}}%
z^{v_{n}}, \label{D}%
\end{equation}
with $c_{n}\in\mathbb{C}^{\ast}$ and $0\leq v_{n}\leq\delta_{n}.$ Since
$a(0)=0,$ putting $z=0$ in (\ref{Delt}) yields by (\ref{G6})
\[
C^{(n)}(0)b(0)=G_{22}^{(n)}(0)b(0)=\left(  -1\right)  ^{n}b(0).
\]
As $a(z)$ and $b(z)$ are coprime, $b(0)\neq0$ and so $C^{(n)}(0)=\left(
-1\right)  ^{n}.$ By (\ref{D}), this yields%
\[
\frac{\Delta_{n}(z)}{\kappa^{(n)}(z)}=C^{(n)}(z)=\left(  -1\right)  ^{n}.
\]
Consequently
\begin{equation}
\kappa^{(n)}(z)=\left(  -1\right)  ^{n}\Delta_{n}(z)\qquad\left(
a(0)=0\right)  . \label{Case2}%
\end{equation}
Now in both cases we see by (\ref{Case1}) and (\ref{Case2}) that $\left\vert
\kappa^{(n)}(1)\right\vert =1.$ Substituting $z=1$ into (\ref{Eq1}) and
(\ref{Eq2}), we get
\begin{align}
\left\vert b(1)\right\vert  &  =\left\vert b(1)G_{11}^{(n)}(1)+a(1)G_{21}%
^{(n)}(1)\right\vert ,\label{Eq3}\\
\left\vert a(1)\right\vert  &  =\left\vert b(1)G_{12}^{(n)}(1)+a(1)G_{22}%
^{(n)}(1)\right\vert . \label{Eq4}%
\end{align}
However we have $a_{t}(2^{k};1)=1$ and $a_{t}(2^{k}-1;1)=k$ by $\left(
\ref{at1}\right)  $ and so by (\ref{G1})-(\ref{G4})
\begin{align*}
G_{11}^{(n+1)}(1)  &  =G_{21}^{(n)}(1),\\
G_{12}^{(n+1)}(1)  &  =G_{22}^{(n)}(1),\\
G_{21}^{(n+1)}(1)  &  =G_{11}^{(n)}(1)-kG_{21}^{(n)}(1),\\
G_{22}^{(n+1)}(1)  &  =G_{12}^{(n)}(1)-kG_{22}^{(n)}(1).
\end{align*}
Hence for every $i=1,2$ and $j=1,2$ we get the recurrence%
\[
G_{ij}^{(n+1)}(1)=-kG_{ij}^{(n)}(1)+G_{ij}^{(n-1)}(1)\qquad\left(
n\geq2\right)  .
\]
The characteristic equation $X^{2}+kX-1=0$ has the roots%
\[
\gamma=\frac{\sqrt{k^{2}+4}-k}{2},\quad\delta=-\frac{\sqrt{k^{2}+4}+k}{2}.
\]
Thus there exists real constants $c_{ij},$ $d_{ij}$ such that%
\[
G_{ij}^{(n)}(1)=c_{ij}\gamma^{n}+d_{ij}\delta^{n}\qquad\left(  n\geq1\right)
.
\]
Substituting these into (\ref{Eq4}) and (\ref{Eq3}), we get%
\begin{align*}
&  \left\vert a(1)\right\vert \left.  =\right.  \left\vert b(1)\left(
c_{12}\gamma^{n}+d_{12}\delta^{n}\right)  +a(1)\left(  c_{22}\gamma^{n}%
+d_{22}\delta^{n}\right)  \right\vert ,\\
&  \left\vert b(1)\right\vert \left.  =\right.  \left\vert b(1)\left(
c_{11}\gamma^{n}+d_{11}\delta^{n}\right)  +a(1)\left(  c_{21}\gamma^{n}%
+d_{21}\delta^{n}\right)  \right\vert .
\end{align*}
Noting that $0<\gamma<1$ and $\left\vert \delta\right\vert >1,$ we deduce that
$a(1)=b(1)=0$ by letting $n\rightarrow\infty$. This is impossible since $a(z)$
and $b(z)$ are coprime, and (iii) is not true.\smallskip
\end{proof}

\textbf{Proof of Theorem \ref{ThStern}.}\quad Theorem \ref{ThStern} is an
immediate consequence of \cite[Th. 4.2.1]{Nishi}. Indeed, the series
$f_{1}(z)$ and $f_{2}(z)\in\mathbb{Q}\left[  \left[  z\right]  \right]  $
converge in $\left\vert z\right\vert <1$ and satisfy the functional equation
$\left(  \ref{Mat}\right)  .$ Moreover, $\alpha$ is a nonzero algebraic number
and so $\alpha^{t^{nk}}$ is not a pole of $A(z)$ for all $n\geq1$ by Lemma
\ref{LemStern}. Since $f_{1}(z)$ and $f_{2}(z)$ are algebraically independent
over $\mathbb{C}\left(  z\right)  $ by Lemma \ref{LemStern3}, $f_{1}(\alpha)$
and $f_{2}(\alpha)$ are algebraically independent by \cite[Th. 4.2.1]{Nishi}.

\end{document}